\documentclass[11pt]{article}
 \usepackage[top=1in,bottom=1in,left=1.5in,right=1.5in]{geometry}
 \usepackage{amsmath}
  \usepackage[dvips]{pstricks} 
  \usepackage{pst-node}
  \usepackage[dvips]{graphicx}
  \usepackage[standard,amsmath,thmmarks]{ntheorem}
  \usepackage{mathrsfs}
  \renewcommand{\geq}{\geqslant}
  \renewcommand{\leq}{\leqslant}
  
  \newcommand{\stack}[2]{\genfrac{}{}{0pt}{}{#1}{#2}}
  \newcommand{\norm}[1]{\Vert #1 \Vert}
  \newcommand{\lm}{\lambda_\text{max}}
  \newcommand{\twobyone}[2]{\bigl( \begin{smallmatrix} #1 \\ #2 \end{smallmatrix} \bigr)}
  \newcommand{\twobytwo}[4]{\bigl( \begin{smallmatrix} #1 & #2 \\ #3 & #4 \end{smallmatrix} \bigr)}

  \newcommand{\N}{\mathbb{N}}
  \newcommand{\R}{\mathbb{R}}
  
  \renewcommand{\i}{\mathbf{i}}

  \renewcommand{\a}{\mathbf{a}}
  \renewcommand{\b}{\mathbf{b}}
  \renewcommand{\c}{\mathbf{c}}
  \renewcommand{\d}{\mathbf{d}}

  \renewcommand{\u}{\mathbf{u}}
  \renewcommand{\v}{\mathbf{v}}
  \newcommand{\w}{\mathbf{w}}
  \newcommand{\x}{\mathbf{x}}
  \newcommand{\y}{\mathbf{y}}
  \newcommand{\z}{\mathbf{z}}
  \newcommand{\0}{\mathbf{0}}
  \newcommand{\1}{\mathbf{1}}

  \newcommand{\lamm}{\lambda_{\text{max}}}

  \newcommand{\cE}{\mathcal{E}}

  \newcommand{\cK}{\mathcal{K}}

  \newcommand{\ms}[1]{\mathscr{#1}}

  \newcommand{\hs}{\hspace*{\parindent}}

  \newcommand{\tr}{\mathop{\mathrm{trace}}\nolimits}

  \newcommand{\trans}{^\top}

  \newcommand{\conv}{\mathrm{conv\;}}

  \newcommand{\rC}{\mathrm{C}}
  \newcommand{\rS}{\mathrm{S}}

  \newcommand{\set}[1]{\{#1\}}
  \newcommand{\bm}[1]{\mbox{{\boldmath $#1$}}}
  
  \newcommand{\floor}[1]{\left\lfloor #1 \right\rfloor}
  \newtheorem{theo}{\bfseries \hs Theorem}[section]
  \newtheorem{defn}[theo]{\bfseries \hs Definition}
  \newtheorem{prob}[theo]{\bfseries \hs Problem}
  \newtheorem{prop}[theo]{\bfseries \hs Proposition}
  \newtheorem{lem}[theo]{\bfseries \hs Lemma}
  \newtheorem{corol}[theo]{\bfseries \hs Corollary}
  \newtheorem{con}[theo]{\bfseries \hs Conjecture}

  \numberwithin{equation}{section} 

\begin{document}
  \title{On the First Eigenvalue of Bipartite Graphs}
  \author{
  Amitava Bhattacharya
  \\
  \texttt{amitava@math.uic.edu}
  \and
  Shmuel Friedland\footnote{
   Visiting Professor, Fall 2007 - Winter 2008,
  Berlin Mathematical School, Berlin, Germany
 }
  \\
  \texttt{friedlan@uic.edu}
  \and
  Uri N. Peled
  \\
  \texttt{uripeled@uic.edu}
  }
  \date{Department of Mathematics, Statistics, and Computer Science,\\
        University of Illinois at Chicago\\
        Chicago, Illinois 60607-7045, USA\\
        September 09, 2008}

 \maketitle
 \begin{abstract}
 In this paper we study the maximum value of the largest eigenvalue
 for simple bipartite graphs, where the number of edges is given
 and the number of vertices on each side of the bipartition is
 given. We state a conjectured solution, which is an analog of the Brualdi-
 Hoffman conjecture for general graphs, and prove
 the conjecture in some special cases.

 \end{abstract}



 \section{Introduction}\label{sec:intro}
 The purpose of this paper to study the maximum value of
 the maximum eigenvalue of certain classes of bipartite graphs.
 These problem are analogous to the problems considered in the
 literature for general graphs and $0-1$ matrices
 \cite{BH, Fr1, Fr2, Row, Sta}.  We describe briefly the main
 problems and results obtained in this paper.

  We consider only finite simple undirected graphs bipartite graphs $G$.
  Let $G=(V\cup W,E)$, where $V=\set{v_1,\ldots,v_m}, W=\set{w_1,\ldots,w_n}$
  are the two set of vertices of $G$.  We view the undirected edges $E$
  of $G$ as a subset of $V\times W$.  Denote by $\deg v_i, \deg w_j$
  the degrees of the vertices $v_i, w_j$ respectively.  Let
  $D(G)  =\set{d_1(G) \geq d_2(G) \geq \cdots \geq d_m(G)}$ be the rearranged
  set of the degrees $\deg v_1,\ldots,\deg v_m$.
  Note that $e(G)=\sum_{i=1}^m \deg v_i$
  is the number of edges in $G$.  Denote by $\lamm(G) $ the maximal
  eigenvalue of $G$.  Denote by $G_{\textrm{ni}}$ the induced
  subgraph of $G$ consisting of nonisolated vertices of $G$.
  Note that $e(G)=e(G_{\textrm{ni}}),
  \lamm(G)=\lamm(G_{\textrm{ni}})$.
  It is straightforward to show, see Proposition \ref{upestimlm}, that

  \begin{equation}\label{basinbge}
  \lamm(G)\leq \sqrt{e(G)}.
  \end{equation}
  Furthermore the equality holds if and only if $G_{\textrm{ni}}$ is
  a complete bipartite graph.  In what follows we assume that
  $G=G_{\textrm{ni}}$, unless stated otherwise.

  The majority of this paper is devoted to refinements of
  (\ref{basinbge}) for noncomplete bipartite graphs.  We now
  state the basic problem that this paper deal with.
 Denote by $K_{p,q}=(V\cup W,E)$ the complete bipartite graph where $\#V=p,\#W=q, E=V\times W$.
 We assume here the normalization $1\le
 p\leq q$.  Let $e$ be a positive integer satisfying $e\leq pq$.
 Denote by $\cK(p,q,e)$ the family of subgraphs $K_{p,q}$ with
 $e$ edges and with no isolated vertices and \emph{which are not complete
 bipartite graphs}.

  \begin{prob}\label{bipprob}  Let $2\leq p\leq q, 1 <e<pq$ be
  integers.
  Characterize the graphs which solve the maximal problem
  \begin{equation}\label{maxfeigp}
  \max_{G\in\cK(p,q,e)} \lm(G).
  \end{equation}

  \end{prob}

  We conjecture below an analog of the Brualdi-Hoffman conjecture for
  nonbipartite graphs \cite{BH}, which was proved by
  Rowlinson \cite{Row}.  See \cite{Fr2, Sta} for the proof of
  partial cases of this conjecture.

  \begin{con}\label{bipcone} Under the assumptions of Problem \ref{bipprob}
  an extremal graph that solves the maximal problem (\ref{maxfeigp})
  is obtained from a complete bipartite graph by adding
  one vertex and a corresponding number of edges.
  \end{con}

  Our first result toward the solution of Problem \ref{bipprob} is
  of interest by itself.  Let  $D  =\set{d_1 \geq d_2 \geq \cdots \geq d_m}$
  be a set of positive integers, and let
  $\ms{B}_D$ be the class of bipartite graphs $G$ with no isolated
  vertices, where $D(G)=D$.  We show that $\max_{G\in \ms{B}_D } \lamm(G)$
  is achieved for a unique graph, up to isomorphism,
  which is the chain graph \cite{Yannakakis}, or the difference graph
  \cite{Mahadev-Peled},
  corresponding to $D$.  (See \S\ref{sec:prelim}.)
  It follows that an extremal graph solving the Problem \ref{bipprob} is a
  chain graph.

  Our main result, Theorem~\ref{thm: prfmconj}, shows that Conjecture~\ref{bipcone}
  holds in the following cases.  Fix $r \geq 2$ and assume that $e \equiv r-1$
  mod $r$. Assume that $l = \floor{\frac{e}{r}} \geq r$.  Let
  $p \in[r,l+1]$ and $q \in [l+1, l + 1 + \frac{l}{r-1}]$.
  So $K_{p,q}$ has more than $e$ edges.
  Then the maximum
  (\ref{maxfeigp}) is achieved
  if and
  only if $G$ is isomorphic to the following chain graph $G_{r,l+1}$.
  $G_{r,l+1}$ obtained from
  $K_{r-1,l+1} = (V \cup W, E)$
  by adding an additional vertex $v_r$ to the set $V$, and connecting $v_r$ to
  the  vertices $w_1, \ldots, w_l$ in $W$.

  We now list briefly the contents of the paper.
  \S2 is a preliminary section in which we recall some known results
  on bipartite graphs and related results on nonnegative matrices.
  In \S3 we show that the maximum eigenvalue of a bipartite graph
  increases if we replace it by the corresponding chain graph.
  \S4 gives upper estimates on the maximum eigenvalue of chain
  graphs.  In \S5 we discuss a minimal problem related to the sharp
  estimate of chain graphs with two different degrees.
  \S6 discuses a special case of the above minimal problem over the
  integers.  In \S7 we introduce $C$-matrices,
  which can be viewed as continuous
  analogs of the square of the adjacency matrix of chain graphs.
  in \S8 prove Theorem \ref{thm: prfmconj}.

 \section{Preliminaries}\label{sec:prelim}
  \begin{figure}[ht]\label{fig:chaingraph}
    \caption{The chain graph $G_D$ for $D=\set{5,2,2,1}$.}
     \begin{center}
      \begin{pspicture}(8,4)
      \psset{radius=.1}
      \Cnode(0,3.5){v1}
      \Cnode(2,3.5){v2}
      \Cnode(4,3.5){v3}
      \Cnode(6,3.5){v4}
      \Cnode(0,0.5){w1}
      \Cnode(2,0.5){w2}
      \Cnode(4,0.5){w3}
      \Cnode(6,0.5){w4}
      \Cnode(8,0.5){w5}
      \ncline{v1}{w1}
      \ncline{v1}{w2}
      \ncline{v1}{w3}
      \ncline{v1}{w4}
      \ncline{v1}{w5}
      \ncline{v2}{w1}
      \ncline{v2}{w2}
      \ncline{v3}{w1}
      \ncline{v3}{w2}
      \ncline{v4}{w1}
      \rput[t](0,4){$v_1$}
      \rput[t](2,4){$v_2$}
      \rput[t](4,4){$v_3$}
      \rput[t](6,4){$v_4$}
      \rput[t](0,0.1){$w_1$}
      \rput[t](2,0.1){$w_2$}
      \rput[t](4,0.1){$w_3$}
      \rput[t](6,0.1){$w_4$}
      \rput[t](8,0.1){$w_5$}
    \end{pspicture}
   \end{center}
  \end{figure}
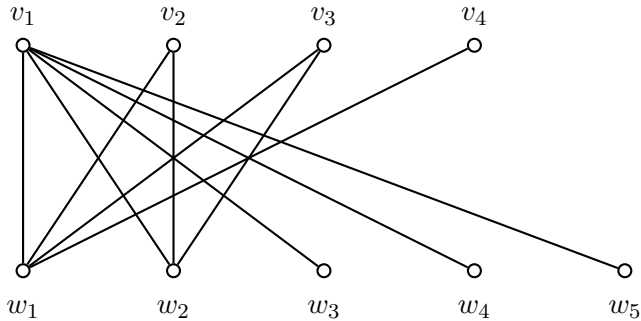

  We now set up some notation and review basic results.
  Denote by $\R^{m\times n}$ the set of $m\times n$ matrices with
  real entries.  We view $A\in\R^{m\times n}$ as
  $A=(A_{i,j})_{i,j=1}^{m,n}$.  Let $G=(V\cup W,E)$ be a bipartite
  graph with $V=\set{v_1,\ldots,v_m},W=\set{w_1,\ldots,w_n}$,
  possibly with isolated vertices.
  We arrange the vertices $V\cup W$ in the order
  $v_1,\ldots,v_m,w_1,\ldots,w_n$.  Then the adjacency  matrix
  $B$ of $G$ is of the form

  \begin{equation}\label{adjmat}
   B = \begin{pmatrix} 0 & A \\ A\trans & 0\end{pmatrix},
  \end{equation}
  where $A$ is an $m \times n$ matrix of $0$'s and $1$'s.
  We call $A$ the \emph{representation matrix} of the bipartite
  graph $G$.  Note that $i-th$ row sum of $A$ is $\deg v_i$ and the
  $j-th$ column sum of $A$ is $\deg w_j$.
  The graph $G$ can be specified
  by specifying the matrix $A$.  Then $G$ does not have isolated
  vertices if and only if $A$ does not zero rows and columns.

  Given $D  =\set{d_1 \geq d_2 \geq \cdots \geq d_m}$, a set of
  positive integers,
  we construct from $D$ the following graph $G_D \in \ms{B}_D$,
  well-known as a chain graph \cite{Yannakakis} or a difference graph \cite{Mahadev-Peled}.
  The vertices of
  $G_D$ are partitioned into $\set{v_1,\ldots,v_m}$ and $\set{w_1,\ldots,w_n}$, $n = d_1$,
  and the neighbors of $v_i$ are $w_1,w_2,\ldots,w_{d_i}$. This is
  illustrated in Figure 1. 

  We now recall the well known spectral properties of the symmetric matrix $B\in \R^{(m+n)\times (m+n)}$
  of the form (\ref{adjmat}), where $A\in \R^{m\times n}_+$, i.e. $A$
  is $m\times n$ matrix with nonnegative entries.
  The spectrum of $B$ is real (by the symmetry of $B$) and symmetric
  around the origin (because if $(\x,\y)$ is an eigenvector for
  $\lambda$, then $(\x,-\y)$ is an eigenvector for $-\lambda$).
  Every real matrix possesses a singular value decomposition (SVD).
  Specifically, if $A$ is $m \times n$ of rank $r$, then there exist
  positive numbers $\sigma_i = \sigma_i(A)$, $i=1,\ldots,r$ (the
  singular values of $A$) and orthogonal matrices $U,V$  of orders
  $m,n$ such that $A = U \Sigma V\trans$, where $\Sigma =
  \text{diag}(\sigma_1,\ldots,\sigma_r,0,\ldots)$ is an $m \times n$
  matrix having the $\sigma_i$ along the main diagonal and otherwise
  zeros. It is possible and usually done to have the $\sigma_i$ in
  non-increasing order. For symmetric matrices the singular values
  are the absolute values of the eigenvalues. The matrix $B$
  from~(\ref{adjmat}) satisfies $B^2 = \twobytwo{A
  A\trans}{0}{0}{A\trans A}$, and so the eigenvalues of $B^2$ are
  those of $AA\trans$ together with those of $A\trans A$. Using the
  SVD for $A$ we see that $AA\trans$ has the $m$ eigenvalues
  $\sigma_1^2,\ldots,\sigma_r^2,0,0,\ldots$ and $A\trans A$ has the
  $n$ eigenvalues $\sigma_1^2,\ldots,\sigma_r^2,0,0,\ldots$. The
  eigenvalues of $B$ are therefore square roots of these numbers,
  and by the symmetry of the spectrum of $B$, the eigenvalues of $B$
  are the $m+n$ numbers $\sigma_1,
  \ldots,\sigma_r,0,\ldots,0,-\sigma_r,\ldots,-\sigma_1$. In
  particular, \emph{the largest eigenvalue of $B$ is $\sigma_1(A)$}.
  We denote this eigenvalue by $\lm(B) = \sigma_1(A)$.  If $B$ is
  the adjacency matrix of $G$ then $\lm(G)=\lm(B)=\sigma_1(A)$.

  For $\x=(x_1,\ldots,x_n)\trans\in\R^n$ we denote by $\norm{\x}=\sqrt{\sum_{j=1}^n
  x_j^2}$, the  Euclidean  norm of $\x$.  For $A\in\R^{m\times n}$ the  operator
  norm of $A$ is given by $\sigma_1(A)=\sqrt{\lm(AA\trans)}=
  \sqrt{\lm(A\trans}A)$.
  We can find $\sigma_1(A)$ by the following maximum principle.
  %
  \begin{equation}\label{quotient}
    \sigma_1(A) = \max_{\stack{\x \in \R^m, \norm{\x} = 1}{\y \in \R^n, \norm{\y} = 1}} \x\trans A
    \y = \max_{\y \in \R^n, \norm{\y}=1} \norm{A\y}.
  \end{equation}
  To see this, consider the SVD $A = U \Sigma V\trans$. Every $\x
  \in \R^m$ with $\norm{\x} = 1$ can be written as $\x = U \a$, $\a
  = (a_1,\ldots,a_m)\trans$, with $\norm{\a} = 1$, and every $\y \in
  \R^n$ with $\norm{\y} = 1$ can be written as $\y = V \b$, $\b =
  (b_1,\ldots,b_n)\trans$, with $\norm{\b} = 1$. Then
  \begin{multline*}
   \x\trans A \y = \a\trans U\trans AV \b
   = \a\trans \Sigma \b
   = \sum_{i=1}^r a_i b_i \sigma_i \leq \sigma_1 \sum_{i=1}^r
   |a_ib_i| \leq
   \\
   \leq \sigma_1 \left(\sum_{i=1}^r a_i^2 \sum_{i=1}^r b_i^2\right)^{\frac{1}{2}}
   \leq \sigma_1 \left(\sum_{i=1}^m a_i^2 \sum_{i=1}^n b_i^2\right)^{\frac{1}{2}}
   = \sigma_1 \norm{\a} \norm{\b} = \sigma_1.
  \end{multline*}
  Equality is achieved when $\x$ is the first column of $U$ and $\y$
  is the first column of $V$, and this proves the first equality
  of~(\ref{quotient}). The second equality is obtained by observing
  that for a given $\y$, the maximizing $\x$ is parallel to $A\y$.

  Another useful fact that can be derived from the SVD $A=U \Sigma
  V\trans$ is the following: if $(\x,\y)$ is an eigenvector of
  $\twobytwo{0}{A}{A\trans}{0}$ belonging to $\sigma_1 > 0$, then $\norm{\x} =
  \norm{\y}$. To see this, observe that $A\y = \sigma_1 \x$ and $A\trans \x = \sigma_1
  \y$. Define vectors $\a = U\trans \x$ and $\b = V\trans \y$. Then
  \begin{gather*}
    \Sigma \b =  \Sigma V\trans \y = U\trans A \y = \sigma_1 U\trans \x =
    \sigma_1 \a
    \\
    \Sigma\trans \a = \Sigma\trans U\trans  \x = V\trans A\trans \x
    = \sigma_1 V\trans \y = \sigma_1 \b.
  \end{gather*}
  It follows that for all $i$ we have $\sigma_i b_i = \sigma_1 a_i$
  and $\sigma_i a_i = \sigma_1 b_i$. Thus $(\sigma_1 + \sigma_i)(b_i-a_i) =
  0$. Since $\sigma_1 + \sigma_i> 0$, it follows that $a_i = b_i$ for
  all $i$, and so $\a = \b$, i.e., $U\trans \x = V\trans \y$. The
  orthogonal matrices $U\trans$ and $V\trans$ preserve the norms,
  and therefore $\norm{\x} = \norm{\y}$.

  Recall the Rayleigh quotient characterization of the largest
  eigenvalue of a symmetric matric $M\in\R^{m\times m}$: $\lambda_{\text{max}}(M) =
  \max_{\norm{\x}=1} \x\trans M \x$. Every $\x$ achieving the
  maximum is an eigenvector of $M$ belonging to
  $\lambda_{\text{max}}(M)$. If the entries of $M$ are non-negative
  ($M \geq 0$), the maximization can be restricted to vectors $\x$
  with non-negative entries ($\x \geq \0$) because $\x\trans M \x
  \leq |\x|\trans M |\x|$ and $\norm{|\x|} = \norm{\x}$, where $|\x|=(|x_1|,
  \ldots,|x_m|)\trans$.

  Recall that a square non-negative matrix $C$ is said to be irreducible
  when some power of $I+C$ is positive (has positive
  entries). Equivalently the digraph induced by $C$ is strongly
  connected. Thus a symmetric non-negative matrix $B$ is irreducible
  when the graph induced by $B$ is connected. For a rectangular
  non-negative matrix $A$, $AA\trans$ is irreducible if and only if
  the bipartite graph with adjacency matrix $B$ given
  by~(\ref{adjmat}) is connected.

  If a symmetric non-negative matrix $B$ is irreducible, then the
  Perron-Frobenius theorem implies that the spectral radius of $B$
  is a simple root of the characteristic polynomial of $B$ and the
  corresponding eigenvector can be chosen to be positive.
  The following result is well known and we bring its proof for
  completeness.

  \begin{prop}\label{upestimlm}  $A=(A_{i,j})_{i,j=1}^{m,n}\in \R^{m\times n}_+$ and assume
  that $B$ is of the form (\ref{adjmat}).  Then
  \begin{equation}\label{upestimlm1}  \lm(B)\leq \sqrt{\sum_{i,j=1}^{m,n}
  A_{i,j}^2}.
  \end{equation}
  Equality holds if and only if either $A=0$ or $A$ is a rank one matrix.
  In particular, if $G$ is a bipartite graph with $e(G)\geq 1$ edges then
  \begin{equation}\label{upestimlm2}  \lm(G)\leq \sqrt{e(G)},
  \end{equation}
  and equality holds if and only if $G_{\rm ni}$ is $K_{p,q}$, where
  $pq=e(G)$.

  \end{prop}
   \begin{Proof}  Let $r$  be the rank of $A$.  Recall that the
   positive eigenvalues of $AA\trans$ are
   $\sigma_1(A)^2,\ldots,\sigma_r(A)^2$.  Hence
   $\tr A A\trans=\sum_{i,j}^{m,n} A_{i,j}^2 =\sum_{k=1}^r
   \sigma_k(A)^2\geq \sigma_1(A)^2$.  Combine this equality with the equality $\lm (B) =\sigma_1(A)$
   to deduce (\ref{upestimlm1}). Clearly, equality holds if and only if either
   $r=0$, i.e. $A=0$, or $r=1$.

   Assume now that $G$ is a bipartite graph.  Let $A$ be the
   representation matrix of $G$.  Then $\tr A A\trans=e(G)$.  Hence
   (\ref{upestimlm1}) implies (\ref{upestimlm2}).

   Assume that
   $G= K_{p,q}$.  Then the entries of the representation matrix $A$
   consist of all $1$.  So rank of $A$ is one and $e(K_{p,q})=pq$,
   i.e. equality holds in (\ref{upestimlm2}).  Conversely, suppose that
   $\lm(G)=\sqrt{e(G)}$.  Hence
   $\lm(G_{\textrm{ni}})=\sqrt{e(G_{\textrm{ni}})}$.  Let $C\in \R^{p\times q}$
   be the representation matrix of $G_{\textrm{ni}}$.
   Since satisfies equality in (\ref{upestimlm1}) we deduce
   that $C$ is a rank one matrix.
   But $C$ is $0-1$ matrix that does not have a zero row or column.
   Hence all the rows and columns of
   $C$ must be identical.  Hence all the entries of $C$ are $1$,
   i.e. $G_{\textrm{ni}}$ is a complete bipartite graph with $e(G)$
   edges.
   \end{Proof}

  \section{The Optimal Graphs}
  The aim of this section to prove the following theorem.
  \begin{theo}\label{theo:maintheo}
  Let $D  =\set{d_1 \geq d_2 \geq \cdots \geq d_m}$ be a set of
  positive integers.
  Then the chain graph $G_D$ is the unique graph in $\ms{B}_D$, (up to
   isomorphism), which solves the maximum problem $\max_{G\in \ms{B}_D }
   \lamm(G)$.
  \end{theo}
  Let us call a graph $G\in \ms{B}_D$ \emph{optimal} if it solves the
  maximum problem of the above theorem.
  Our first goal is to prove that every optimal graph is connected.
  For that purpose we partially order the finite sets
  of positive integers as follows.
  \begin{defn}\label{def:cover}
  Let $D = \set{d_1 \geq d_2 \geq \cdots \geq d_m}$ and $D' =
  \set{d'_1 \geq d'_2 \geq \cdots \geq d'_{m'}}$ be sets of $m$ and
  $m'$ positive integers. Then $D > D'$ means that $m \geq m'$, and
  $d_1 \geq d'_1$, $d_2 \geq d'_2$, \ldots, $d_{m'} \geq d'_{m'}$,
  and $D \neq D'$.
  \end{defn}
  \begin{theo}\label{theo:covers}
    If $D > D'$, then $\lambda_{\text{max}}(G_D) >
    \lambda_{\text{max}}(G_{D'})$.
  \end{theo}
  \begin{Proof}
    Let $A$ be the $m \times d_1$ matrix of $0$'s and $1$'s with row
    sums $d_1 \geq d_2 \geq \cdots \geq d_m$, and columns ordered so
    that each row is left-justified ($1$'s first, then $0$'s). Then
    $B = \twobytwo{0}{A}{A\trans}{0}$ of order $m+d_1$ is the adjacency matrix of $G_D$.
    Let $A'$ and $B'$ be defined similarly for $G_{D'}$.

    Let $M = BB\trans$ and $M' = B'{B'}\trans$. Then $M$ and
    $M'$ are symmetric non-negative irreducible matrices of orders
    $m$ and $m'$, and $\lambda_{\text{max}}(G_D) =
    \lambda_{\text{max}}(M)$, $\lambda_{\text{max}}(G_{M'}) =
    \lambda_{\text{max}}(M')$.

    \emph{Case 1:} $m=m'$. In this case, by
    Definition~\ref{def:cover}, at least one of the inequalities
    $d_1 \geq d'_1$, $d_2 \geq d'_2$, \ldots, $d_{m} \geq d'_{m}$
    holds with strict inequality. It follows that $M \geq M'$ (i.e.,
    $M-M'$ is a non-negative matrix), and some integer $i\in [1,m]$
    satisfies $M_{i,i} > M'_{i,i}$. Therefore every positive vector
    $\y$ satisfies $\y\trans M \y > \y\trans M '\y$. Let $\y=\x'$ be
    the positive Perron-Frobenius eigenvector of the irreducible
    matrix $M'$, with $\norm{\x'} = 1$. Then by the Rayleigh
    quotient we have
    \[\lambda_{\text{max}}(M) \geq {\x'}\trans M \x' >
    {\x'}\trans M' \x' = \lambda_{\text{max}}(M'),\]
    as required.

    \emph{Case 2:} $m>m'$. In this case, let $L$ be the principal
    submatrix of $M$ consisting of its first $m'$ rows and columns.
    By Definition~\ref{def:cover} we have $d_1 \geq d'_1$, $d_2 \geq
    d'_2$, \ldots, $d_{m'} \geq d'_{m'}$. Therefore $L \geq M'$, and
    hence $\lambda_{\text{max}}(L) \geq \lambda_{\text{max}}(M')$.

    Since $L$ is symmetric and non-negative, there exists a vector
    $\y \in \R^{m'}$ with $\y \geq \0$, $\norm{\y} = 1$ satisfying
    $\lambda_{\text{max}}(L) = \y\trans L \y$. Extend $\y$ with
    zeros to a vector $\x \in \R^m$. Then $\x \geq \0$ and
    $\norm{\x} = 1$ and $\y\trans L \y = \x\trans M \x \leq
    \lambda_{\text{max}}(M)$. Equality cannot occur here, for if it
    did, then $\x$ would be the unique Perron-Frobenius eigenvector
    of the irreducible matrix $M$ and $\x$ would be positive,
    whereas $x_i = 0$ for $i > m'$. Thus $\lambda_{\text{max}}(M) >
    \lambda_{\text{max}}(L) \geq \lambda_{\text{max}}(M')$, as
    required.
  \end{Proof}
  \begin{lem}\label{theo:easylemma}
    If $G \in \ms{B}_D$ is connected, then $\lambda_{\text{max}}(G) \leq
    \lambda_{\text{max}}(G_D)$.
  \end{lem}
  \begin{Proof}
    Let $D = \set{d_1\geq \cdots \geq d_m}$ and $n \geq d_1$. Let $B
    = \twobytwo{0}{A}{A\trans}{0}$ be the adjacency matrix of $G$, where $A$ is $m \times n$
    with row sums given by $D$. Since $G$ is connected, $B$ is
    irreducible. Let $(\x,\y)$ be the positive Perron-Frobenius
    eigenvector of $B$ belonging to $\lambda_{\text{max}}(G) =
    \sigma_1(A)$, with $\x = (x_1,\ldots,x_m)$, $\y =
    (y_1,\ldots,y_n)$:
    \begin{equation}\label{eq:Bxy}
     \sigma_1(A) \twobyone{\x}{\y} =  \twobytwo{0}{A}{A\trans}{0} \twobyone{\x}{\y},
    \end{equation}
    and so $A\y = \sigma_1(A)\x$.

    As we observed in Section~\ref{sec:prelim}, we have $\norm{\x} =
    \norm{\y}$, and so we may choose a normalization such that $\norm{\x} =
    \norm{\y} = 1$.

    We reorder the columns of $A$ so that $y_1 \geq
    y_2 \geq \cdots \geq y_n > 0$.
    The
    rows are still in their original order, and so the row sums are
    $d_1\geq \cdots \geq d_m$ in this order.

    Let $\overleftarrow{A}$ be the matrix obtained from $A$ by
    left-justifying each row, i.e., moving all the $1$'s of the row
    to the beginning of the row. Then
    $\twobytwo{0}{\overleftarrow{A}}{\overleftarrow{A}\trans}{0}$ is
    the adjacency matrix of $G_D$ with $n-d_1$ zero rows and columns
    appended at the end, and therefore $\lambda_{\text{max}}(G_D) =
    \sigma_1(\overleftarrow{A})$.

    Since $y_1 \geq y_2 \geq \cdots \geq y_n \geq 0$ and since
    $\overleftarrow{A}$ is obtained from $A$ by left-justifying each
    row, we have $\overleftarrow{A}\y \geq A\y$. Since $\x \geq \0$,
    we have $\x\trans \overleftarrow{A}\y \geq \x\trans A \y$.
    (\ref{quotient}) yields
    %
    \begin{multline}\label{eqn:betterthanGD}
     \lambda_{\text{max}}(G_D) = \sigma_1(\overleftarrow{A})
     = \max_{\stack{\u \in \R^m, \norm{\u} = 1}{\v \in \R^n, \norm{\v} = 1}}
     \u\trans \overleftarrow{A} \v \geq \x\trans \overleftarrow{A}\y \geq \x\trans A \y
     \\
     = \x\trans \sigma_1(A) \x = \sigma_1(A) = \lambda_{\text{max}}(G).
    \end{multline}
  \end{Proof}
  \begin{lem}\label{theo:connected}
   An optimal graph must be connected.
  \end{lem}
  \begin{Proof}
    Let $G \in \ms{B}_D$ be an optimal graph. The graph $G$ is
    bipartite, and one side of the bipartition (call it the first
    side) has degrees given by $D$. Let $G_1$, \ldots, $G_k$ be the
    connected components of $G$. Then $\lambda_{\text{max}}(G) =
    \lambda_{\text{max}}(G_i)$ for some $i$.

    Like $G$, the component $G_i$ is also bipartite with the
    bipartition inherited from that of $G$. Let $D_i$ be the set of
    degrees of $G_i$ on the first side of the bipartition.

    If $G$ is disconnected, then $D > D_i$, and therefore
    $\lambda_{\text{max}}(G) \geq \lambda_{\text{max}}(G_D) >
    \lambda_{\text{max}}(G_{D_i}) \geq \lambda_{\text{max}}(G_i)$,
    where the first inequality is by the optimality of $G$, the
    second by Theorem~\ref{theo:covers}, and the third by
    Lemma~\ref{theo:easylemma} and the connectivity of $G_i$. This
    contradicts the equality above and proves that $G$ must be connected.
  \end{Proof}
  We are now ready to prove our main theorem.

  \begin{Proof}
    (of Theorem~\ref{theo:maintheo}). Let $G \in \ms{B}_D$ be
    optimal with adjacency matrix $B=\twobytwo{0}{A}{A\trans}{0}$.
    By Lemma~\ref{theo:connected} $G$ is connected. We begin as in
    the proof of Lemma~\ref{theo:easylemma}. We let $(\x,\y)$ be
    the positive Perron-Frobenius eigenvector of $B$ belonging to
    $\lambda_{\text{max}}(G) = \sigma_1(A)$, with $\x =
    (x_1,\ldots,x_m)\trans$, $\y = (y_1,\ldots,y_n)\trans$, $\norm{\x} =
    \norm{\y} = 1$. In other words, (\ref{eq:Bxy}) holds, or
    equivalently
    \begin{gather}\label{eq:Bxyyx}
      A \y = \sigma_1(A) \x
      \\
      A\trans \x = \sigma_1(A) \y
    \end{gather}
    However, this time we reorder both the rows and the columns of
    $A$ so that $x_1 \geq x_2 \geq \cdots \geq x_m > 0$ and $y_1
    \geq y_2 \geq \cdots \geq y_n > 0$, so now the row sums of $A$,
    which we still denote by $d_1,d_2,\ldots,d_m$, are not
    necessarily non-decreasing. As before, we let
    $\overleftarrow{A}$ be the matrix obtained from $A$ by
    left-justifying each row. The graph with adjacency matrix
    $\twobytwo{0}{\overleftarrow{A}}{\overleftarrow{A}\trans}{0}$ is
    still isomorphic to $G_D$ plus $n-d_1$ isolated vertices, and
    therefore $\lambda_{\text{max}}(G_D) =
    \sigma_1(\overleftarrow{A})$. For the same reasons as before we
    have $\overleftarrow{A}\y \geq A\y$, and therefore
    (\ref{eqn:betterthanGD}) holds. Moreover, by the optimality
    of $G$ we have equality throughout~(\ref{eqn:betterthanGD}). In
    particular $G_D$ is optimal and $\lambda_{\text{max}}(G_D) =
    \sigma_1(A)$, so from now on we abbreviate $\sigma_1(A) =
    \sigma_1(\overleftarrow{A})= \sigma_1$. Now $\overleftarrow{A}\y
    \geq A\y$ and $\x\trans \overleftarrow{A}\y = \x\trans A \y$ and
    $\x > \0$ give
    \begin{equation}\label{eq:tightAy}
      \overleftarrow{A}\y = A\y = \sigma_1\x.
    \end{equation}
    The first two rows of~(\ref{eq:tightAy}) and $x_1 \geq x_2$ now give
    \[y_1 + \cdots + y_{d_1} = \sigma_1 x_1 \geq \sigma_1 x_2 =y_1 + \cdots + y_{d_2}, \]
    and since $\y > \0$ we must have $d_1 \geq d_2$. The same
    argument with rows 2 and 3 shows $d_2 \geq d_3$, and so on. We
    have established that the row sums of $A$ are non-decreasing,
    i.e.,
    \begin{equation}\label{eq:nondecreasing}
     d_1 \geq d_2 \geq \cdots \geq d_m.
    \end{equation}
    Note that by~(\ref{eq:nondecreasing}), the columns of
    $\overleftarrow{A}$ are top-justified, i.e., the $1$'s are above
    the $0$'s. For this reason and $\x \geq \0$ we have
    $\overleftarrow{A}\trans \x \geq A\trans\x$, and hence $\y\trans
    \overleftarrow{A}\trans \x \geq \y\trans A\trans\x$ by $\y \geq
    \0$. The analog of~(\ref{eqn:betterthanGD}) for
    $\overleftarrow{A}\trans$ now holds with equality throughout and
    we obtain
    \begin{equation}\label{eq:tightATx}
      \overleftarrow{A}\trans\x = A\trans\x = \sigma_1\y.
    \end{equation}

    Our remaining task is to show that $d_1 = n$ and $A =
    \overleftarrow{A}$, and therefore $G$ is isomorphic to $G_D$.
    For that purpose we need notation for rows of
    $\overleftarrow{A}$ with equal sums, and similarly for columns.

    We introduce the following notation for the row sums of
    $\overleftarrow{A}$:
    \begin{gather}\label{eq:rownotation}
      \begin{split}
       r_1 = d_1 = \cdots = d_{m_1} > r_2 = d_{m_1 + 1} = \cdots = d_{m_1 +
       m_2} > \cdots >
       \\
        > r_h = d_{m_1 + \cdots + m_{h-1}+1} = \cdots = d_{m_1 + \cdots +
        m_h},
      \end{split}
      \\
      \intertext{where}
      m_1 + \cdots + m_h = m. \notag
    \end{gather}
    This is illustrated in Figure~\ref{fig:Ferrers}.
    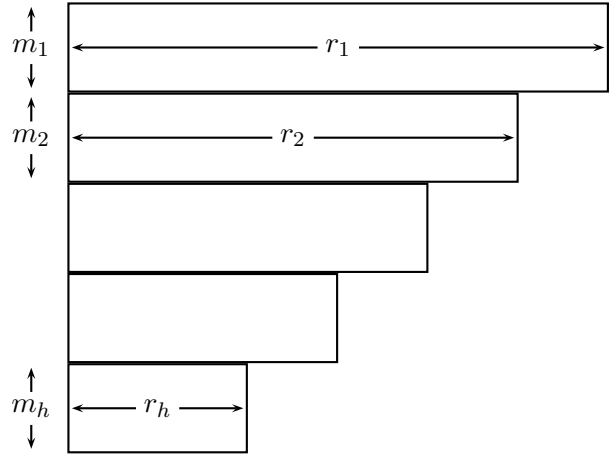
\begin{figure}[ht]
      \caption{The notation for the row sums of $\protect\overleftarrow{A}$.}
      \label{fig:Ferrers}
       \begin{center}
        \psset{unit=1.2cm}
        \begin{pspicture}(7,5)
         \psframe(1,0)(3,1)
         \psframe(1,1)(4,2)
         \psframe(1,2)(5,3)
         \psframe(1,3)(6,4)
         \psframe(1,4)(7,5)
         \psline{<->}(1.05,.5)(2.95,.5)
         \rput*(2,.5){$r_h$}
         \psline{<->}(1.05,3.5)(5.95,3.5)
         \rput*(3.5,3.5){$r_2$}
         \psline{<->}(1.05,4.5)(6.95,4.5)
         \rput*(4,4.5){$r_1$}
         \psline{<->}(0.6,.05)(0.6,.95)
         \rput*(0.6,.5){$m_h$}
         \psline{<->}(0.6,3.05)(0.6,3.95)
         \rput*(0.6,3.5){$m_2$}
         \psline{<->}(0.6,4.05)(0.6,4.95)
         \rput*(0.6,4.5){$m_1$}
      \end{pspicture}
     \end{center}
    \end{figure}

    From~(\ref{eq:tightAy}) we have $\sigma_1 x_i = (\overleftarrow{A} \y)_i = y_1 + \cdots +
    y_{d_i}$. Therefore by~(\ref{eq:rownotation}) and $\y > \0$ we
    obtain
    \begin{equation}\label{eq:xordered}
      \begin{split}
        x_1 = \cdots = x_{m_1} > x_{m_1+1} = \cdots = x_{m_1+m_2} >
        \cdots >
        \\
        x_{m_1+ \cdots + m_{h-1}+1} = \cdots = x_{m_1 + \cdots +
        m_h} > 0.
      \end{split}
    \end{equation}
    Analogously using~(\ref{eq:tightATx}) and (\ref{eq:rownotation})
    and $\x > \0$ we obtain
    \begin{equation}\label{eq:yordered}
      \begin{split}
        y_1 = \cdots = y_{r_h} > y_{r_h+1} = \cdots = y_{r_{h-1}} > \cdots >
        \\
        y_{r_2 + 1} = \cdots = y_{r_1} > 0
        = y_{r_1+1} = y_{r_1+2} = \cdots = y_n.
      \end{split}
    \end{equation}
    From~(\ref{eq:yordered}) and $\y > \0$, we conclude that
    \[d_1 = r_1 = n.\]

    We are now ready to show that $A = \overleftarrow{A}$. Since $d_1 =
    r_1 = n$, the first $m_1$ rows of $A$ are all-$1$, and so are
    the first $m_1$ rows of $\overleftarrow{A}$. Now let $m_1 + 1
    \leq i \leq m_1 + m_2$ be an index of one of the next $m_2$
    rows. Both $A$ and $\overleftarrow{A}$ have $d_i = r_2$ $1$'s in
    row $i$. Let the $1$'s in row $i$ of $A$ lie in columns
    $k_1,\ldots,k_{r_2}$. Then by~(\ref{eq:tightAy}) we have
    \begin{equation}\label{eq:secondrowscoincide}
     \sum_{j=1}^{r_2} y_{k_j} = (A\y)_i =
     (\overleftarrow{A}\y)_i = \sum_{j=1}^{r_2} y_j.
    \end{equation}
    However, by~(\ref{eq:yordered}) the last $r_1 - r_2$ components
    of $\y$ are smaller than all other components. Therefore if any
    $k_j$ lies in the range $\set{r_2 + 1,\ldots, r_1}$, it would
    follow that $\sum_{j=1}^{r_2} y_{k_j} < \sum_{j=1}^{r_2} y_j$,
    contradicting~(\ref{eq:secondrowscoincide}). Therefore $k_j = j$
    for $j = 1 \ldots, r_2$, in other words rows $i$ of $A$ and
    $\overleftarrow{A}$ are the same.

    An analogous argument can be applied to the next $m_3$ rows, and
    so on, and it follows that $\overleftarrow{A} = A$.
  \end{Proof}

   The arguments of the proof of the above theorem yield.
   \begin{corol}\label{cor:optsolmpr}  Let the assumptions of
   Problem \ref{bipprob} holds.  Then any $H\in\cK(p,q,e)$
   satisfying  $\max_{G\in\cK(p,q,e)} \lm(G)=\lm(H)$ is isomorphic
   to $G_D$, for some $D  =\set{d_1 \geq d_2 \geq \cdots \geq d_m}$,
   where $m\leq p$ and $d_1\leq q$.

   \end{corol}
  \section{Estimations of the Largest Eigenvalue}
  In this section we give lower and upper bounds for
  $\lambda_{\text{max}}(G)$, where $G$ is an optimal graph with
  a given adjacency matrix $\twobytwo{0}{A}{A\trans}{0}$.
  (Our upper bound improves the upper bound (\ref{upestimlm2}).)

  Recall the concept of the second compound matrix $\Lambda_2 A$ of
  an $m \times n$ matrix $A = (A_{i,j})$ \cite{Gantmacher}: $\Lambda_2 A$ is
  an $\binom{m}{2} \times
  \binom{n}{2}$ matrix with rows indexed by $(i_1,i_2)$, $1 \leq i_1 < i_2 \leq m$
  and columns indexed by $(j_1,j_2)$, $1 \leq j_1 < j_2 \leq n$. The
  entry in row $(i_1,i_2)$ and column $(j_1,j_2)$ of $\Lambda_2 A$ is given by
  \begin{equation}\label{Lambda2}
    \Lambda_2A_{(i_1,i_2)(j_1,j_2)} = \det
    \begin{pmatrix}
      A_{i_1,j_1} A_{i_1,j_2}
      \\
      A_{i_2,j_1} A_{i_2,j_2}
    \end{pmatrix}.
  \end{equation}
  Note that $(\Lambda_2 A)\trans=\Lambda_2 A\trans$.
  It follows from the Cauchy-Binet theorem that for matrices $A,B$
  of compatible dimensions one has $\Lambda_2 (AB) = (\Lambda_2 A)
  (\Lambda_2 B)$. One also has $\Lambda_2 I = I$ and therefore
  $\Lambda_2 A^{-1} = (\Lambda_2 A)^{-1}$ for nonsingular $A$. In
  particular, the second compound matrix of an orthogonal matrix is
  orthogonal, and therefore the SVD carries over to the second
  compound: if the SVD of $A$ is $A = U \Sigma V\trans$, then the
  SVD of $\Lambda_2 A$ is $(\Lambda_2 A) = (\Lambda_2 U) (\Lambda_2
  \Sigma) (\Lambda_2 V)\trans$. It follows that if the singular
  values of $A$ are $\sigma_1 \geq \sigma_2 \geq \cdots$, then the
  singular values of $\Lambda_2A$ are $\sigma_i \sigma_j$, $i < j$.
  In particular, when the rank of $A$ is larger than $1$,
  equivalently $\Lambda_2A \neq 0$, we have
  \begin{equation}\label{sigma1sigma2}
    \sigma_1 \sigma_2 = \sigma_1(\Lambda_2 A) = \max_{\w \neq \0} \frac{\norm{(\Lambda_2
    A)\w}}{\norm{\w}},
  \end{equation}
  where the second equality follows by applying~(\ref{quotient}) to $\Lambda_2
  A$.

  We now specialize to $A$ given by~(\ref{adjmat}), which is the
  adjacency matrix of an optimal graph. Thus $A$ is a matrix of
  $0$'s and $1$'s whose rows are left-justified and whose columns
  are top-justified. We use the notation~(\ref{eq:rownotation}) for
  the row sums of $A$. For such $A$ the entries of $\Lambda_2 A$ can
  only be $0$ or $-1$. Indeed, if in~(\ref{Lambda2}) $A_{i_2,j_2} =
  1$, then $A_{i_1,j_2} = A_{i_2,j_1} = A_{i_1,j_1} = 1$ and the
  determinant vanishes. If $A_{i_2,j_2} = 0$ and the determinant does
  not vanish, then again $A_{i_1,j_2} = A_{i_2,j_1} = A_{i_1,j_1} = 1$
  and the determinant equals $-1$. In the latter case we say that
  $(i_1,i_2)$ and $(j_1,j_2)$ are in a \emph{$\Gamma$-configuration}.

  To estimate $\sigma_1 \sigma_2$ from below, we take a particular
  column vector $\w$ in~(\ref{sigma1sigma2}): the $(j_1,j_2)$, $j_1
  < j_2$ entry of $\w$ is $1$ if column $(j_1,j_2)$ of $\Lambda_2 A$
  is nonzero; otherwise this entry of $\w$ is zero. (The assumption
  $\Lambda_2 A \neq 0$ implies that $\w \neq \0$.)
  By~(\ref{sigma1sigma2}) we have
  \begin{equation}\label{eq:fraction}
   \sigma_1 \sigma_2 \geq \frac{\norm{(\Lambda_2 A)\w}}{\norm{\w}}.
  \end{equation}

  Since $\w$ is a vector of $0$'s and $1$'s, $\norm{\w}^2$ is the
  number of nonzero entries of $\w$, that is to say, the number of
  nonzero columns of $\Lambda_2 A$. We count the nonzero columns
  $(j_1,j_2)$, $j_1 < j_2$ of $\Lambda_2 A$ as follows. Fix $j_2$.
  There is a unique $k = 1,\ldots,h-1$ such that $r_{k+1} + 1 \leq j_2 \leq
  r_k$. If $j_1$ is chosen among $1,\ldots,r_{k+1}$, then there
  exist $(i_1,i_2)$ such that $(i_1,i_2)$ and $(j_1,j_2)$ are in a
  $\Gamma$-configuration, and otherwise not. It follows that
  for our fixed $j_2$, there are $r_{k+1}$ values of $j_1$ such that
  column $(j_1,j_2)$ of $\Lambda_2 A$ is nonzero. We can vary $j_2$
  without changing $k$ in $r_k - r_{k+1}$ ways, so $\Lambda_2 A$ has $r_{k+1}(r_k -
  r_{k+1})$ nonzero columns corresponding to the same $k$. Summing
  over $k$, we conclude that
  \begin{equation}\label{eq:denominator}
    \norm{\w} ^2 = \sum_{k=1}^{h-1} r_{k+1}(r_k - r_{k+1}).
  \end{equation}
  By similar arguments we see that for $1 \leq k < l \leq h$, the
  vector $(\Lambda_2 A)\w$ has $m_k m_l$ entries equal to $-r_l (r_k -
  r_l)$, and that all other entries of $(\Lambda_2 A)\w$ vanish.
  Therefore
  \begin{equation}\label{eq:numerator}
    \norm{(\Lambda_2 A)\w}^2 = \sum_{1 \leq k < l \leq h}
    m_k m_l [r_l (r_k - r_l)]^2.
  \end{equation}
  From~(\ref{eq:fraction}), (\ref{eq:denominator}) and
  (\ref{eq:numerator}), we obtain
  \begin{equation}\label{eq:omega}
    \sigma_1^2 \sigma_2^2 \geq \omega \equiv \frac{\sum_{1 \leq k < l \leq h}
    m_k m_l [r_l (r_k - r_l)]^2}{\sum_{k=1}^{h-1} r_{k+1}(r_k -
    r_{k+1})}.
  \end{equation}
  As we have noted above, we assume that $h > 1$, for otherwise $A$
  has rank $1$. If $h=1$ we define $\omega = 0$.

  We can improve the lower bound~(\ref{eq:omega}) by the following
  consideration. The graph with adjacency matrix
  $\twobytwo{0}{A\trans}{A}{0}$ is isomorphic to the one with
  adjacency matrix $\twobytwo{0}{A}{A\trans}{0}$. Therefore we can
  repeat the work in this section with $A\trans$ replacing $A$. This
  amounts to transposing the Ferrers diagram illustrated in
  Figure~\ref{fig:Ferrers}. Instead of~(\ref{eq:omega}) we now have
  \begin{equation}\label{eq:omegaprime}
    \sigma_1^2 \sigma_2^2 \geq \omega' \equiv \frac{\sum_{1 \leq k < l \leq h}
    m'_k m'_l [r'_l (r'_k - r'_l)]^2}{\sum_{k=1}^{h-1} r'_{k+1}(r'_k -
    r'_{k+1})},
  \end{equation}
  where for $i=1,\ldots,h$ we have $r'_i = m_1+ \cdots + m_{h-i+1}$ and $m'_i =
  r_{h-i+1}-r_{h-i+2}$ ($r_{h+1}=0$).

  Combining~(\ref{eq:omega}) and (\ref{eq:omegaprime}) we obtain
  \begin{equation}\label{eq:omegastar}
    \sigma_1^2 \sigma_2^2 \geq \omega^* \equiv \max \set{\omega,\omega'}.
  \end{equation}

  We are now ready to estimate $\sigma_1^2$ from above.
  \begin{theo}\label{theo:estimate}  Let $D = \set{d_1 \geq d_2 \geq \cdots \geq
  d_m}$ be a set of positive integers, where $d_1>d_m$.  Then
  \begin{equation}\label{upestopg}
  \lm (G_D)^2 \leq \frac{e(G_D) + \sqrt{e(G_D)^2 -4\omega^*(G_D)}}{2},
  \end{equation}
  where
  $\omega^*(G_D)$ is defined
  in~(\ref{eq:omegastar}). Assume that in
  the Ferrers diagram given in
  Figure~\ref{fig:Ferrers} $h=2$. I.e. the degree of the vertices in each group of $G_D$
  have exactly two distinct values.  Then equalities hold in
  (\ref{eq:omega}), (\ref{eq:omegaprime}), (\ref{eq:omegastar}) and (\ref{upestopg}).
  In particular
  \begin{equation}\label{eq:omegeq}
  \omega^*(D)=\omega=\omega'= m_1m_2r_2(r_1-r_2).
  \end{equation}

  \end{theo}
  \begin{Proof}
    Since the eigenvalues of $A\trans A$ are $\lm(G_D)^2=\sigma_1^2,
    \sigma_2^2,\ldots, \sigma_r^2, 0,0,\ldots$, we have $\sum_{i=1}^r
    \sigma_i^2 = \tr A\trans A = \sum_{i,j} A_{i,j}^2 = e$. Let us
    denote $a = \sigma_1^2 + \sigma_2^2$, so that $a \leq e=e(G_D)$, and $b = \sigma_1^2
    \sigma_2^2$, so that $b \geq \omega^*$ by~(\ref{eq:omegastar}). Solving for $\sigma_1^2$
    we obtain $\sigma_1^2 = \frac{a+\sqrt{a^2-4b}}{2} \leq \frac{e+\sqrt{e^2-4\omega^*}}{2}$.

    Assume that $h=2$.  Then $A$ has rank $2$, $\Lambda_2 A$ has
    rank $1$, and $a=e$.  Furthermore the definitions of
    $\omega,\omega'$ yield the equalities (\ref{eq:omegeq}).
    To complete the proof, we show that equality holds
    in~(\ref{eq:fraction}) and therefore also in~(\ref{eq:omega}),
    i.e., $b=\omega$.

    Since $\Lambda_2 A$ has rank $1$ and its elements are only $0$
    and $-1$, all its nonzero rows are equal. Say it has $c$ nonzero
    rows, each with $d$ elements of $-1$. The trace of $(\Lambda_2 A)\trans
    (\Lambda_2 A)$ is the sum of squares of the singular values of $
    \Lambda_2 A$, which equals $(\sigma_1(\Lambda_2 A))^2 =
    \sigma_1^2 \sigma_2^2$ in our case. This trace also equals the
    sum of squares of the elements of $\Lambda_2 A$, namely $cd$.

    On the other hand, our chosen vector $\w$ satisfies $\norm{\w}^2 =
    d$ and $\norm{(\Lambda_2 A)\w}^2 = cd^2$ (because each of the $c$ nonzero rows of $\Lambda_2 A$
    multiplied by $\w$ gives $-d$). Hence $\frac{\norm{(\Lambda_2 A)\w}^2}{\norm{\w}^2} =
    cd$. Thus both sides of~(\ref{eq:fraction}) are equal to $\sqrt{cd}$.
   \end{Proof}
  %
   We suspect that under the conditions of Theorem \ref{theo:estimate} for
   $h\geq 3$ one has strict inequality in (\ref{upestopg}).

  \section{A Minimization Problem}

  The first step in proving Conjecture~\ref{bipcone} is to show its
  validity in the case $h = 2$ in Figure~\ref{fig:Ferrers}.
  We note that for $h=2$, (\ref{upestopg}) is tight by Theorem
  \ref{theo:estimate}.  Theorem \ref{theo:estimate} also
  implies that equality holds in (\ref{eq:omega}), (\ref{eq:omegaprime}) and
  (\ref{eq:omegastar}).  This motivates us to consider the
  problem of minimizing $\omega^*(G)$.

  Let $n_1: = r_2$, $n_2:=r_1 - r_2$.  Then the condition that the
  chain graph $G$ has $e$ edges is equivalent to
  \begin{equation}\label{constre}
    m_1n_1 + m_1n_2 + m_2n_1 = e.
  \end{equation}

  Formula~(\ref{eq:omega}) for the case $h=2$ gives

  %
  \begin{equation}\label{omegarank2}
    \omega = m_1 m_2 n_1 n_2.
  \end{equation}
  Let $\cK_2(p,q,e) \subseteq \cK(p,q,e)$ be set of all subgraphs
  of $K_{p,q}$ isomorphic to some $G_D$ whose
  Ferrers diagram, given in
  Figure~\ref{fig:Ferrers}, satisfies the condition $h = 2$.
  By the above discussion for $h = 2$, the problem of finding
  $\max_{G \in \cK_2(p,q,e)} \lm(G)$ is equivalent to the following
  minimization problem over the integers.

  \begin{prob}\label{prob:minint}
    Let $p$, $q$ and $e$ be integers satisfying $2 \leq p \leq q$ and $3 \leq e < pq$.
    Find the minimum of $m_1 m_2 n_1 n_2$ in positive integers $m_1$, $m_2$, $n_1$ and $n_2$
    satisfying $m_1 + m_2 \leq p$, $n_1+n_2 \leq q$ and the constraint~(\ref{constre}).
  \end{prob}
  Note if $p=q$, then Problem~\ref{prob:minint} remains invariant under the duality of
  exchanging $(m_1,m_2)$ with $(n_1,n_2)$.  Conjecture~\ref{bipcone}
  implies that any minimal solution of Problem~\ref{prob:minint}
  satisfies the condition $\min(m_2,n_2) = 1$.

  In order to prove Conjecture~\ref{bipcone} in the cases discussed
  in Theorem~\ref{thm: prfmconj} we need to consider a problem of minimizing
  $m_1 m_2 n_1 n_2$  under certain constraints, where $m_1$, $m_2$, $n_1$ and $n_2$
  are real numbers.
  We start with the following simple lemma.
  \begin{lem}\label{auxmin}
    Let $b$, $b$ and $e$ be positive real numbers satisfying $e
    > a + b$. Assume that
    \begin{equation}\label{auxminc}
      ax + by = e, \quad 1 \leq x,\;1 \leq y.
    \end{equation}
    Then
    \begin{equation}\label{auxminin}
      xy \geq \min\left(\frac{e-a}{b},
      \frac{e-b}{a}\right). 
    \end{equation}
    Equality holds if and only if
    \begin{enumerate}
      \item $x=1$ when $a<b$;
      \item $y=1$ when $b<a$;
      \item $x=1$ or $y=1$ when $a=b$.
    \end{enumerate}
  \end{lem}

  \begin{proof}
    Set $by = e - ax$ and observe that $f(x) := byx = (e - ax)x$.
    Note that $f$ is a parabola, with its maximum at $x_0 := \frac{e}{2a}$.
    So $f$ is decreasing for $x > x_0$ and increasing for $x < x_0$.
    The minimum of $xy = \frac{f}{b}$, given the constraints $x \geq 1$, $y \geq 1$
    is achieved only when $x = 1$ and $y = \frac{e-a}{b} > 1$ (the minimum
    possible value of $x$), or when $y = 1$ and $x = \frac{e-b}{a} > 1$ (where $x$ is the maximum
    possible value).  For $a < b$ and  $x = 1$ we have
    $xy = \frac{e-a}{b}$, which is the LHS of
    (\ref{auxminin}).  For $b < a$  and
    $y = 1$ we have $xy = \frac{e-b}{a}$, which is the RHS of
    (\ref{auxminin}).
  \end{proof}
  \begin{prop}\label{prop:upbndm}
     Let $2 \leq r, e \in \N$ and assume that
       \begin{equation}\label{asmrcon1}
          e = lr + r-1,  r \leq p, q \leq l + 1 + \frac{l}{r-1}.
       \end{equation}
     Let $G = (V,W,E) \in \cK(p,q,e)$.
     Then $\#V \geq r$ and $\#W \geq r$.
  \end{prop}

  \begin{proof}
     Assume to the contrary that $\#V
     \leq r-1$. Since $G$ is not a complete bipartite graph,
       \[e < (r-1)q  \leq (r-1)\left(l+1+\frac{l}{r-1}\right) = lr+r-1 = e,\]
     which is impossible.  Replacing $q$ by $p$ we deduce that $\#W \geq r$.
  \end{proof}
  Hence if $p$, $q$ and $e$ satisfy (\ref{asmrcon1}), then for $D \in \cK_2(p,q,e)$ we must have
  $m_1 + m_2 \geq r$ and $n_1 + n_2\geq r$. Since $m_1$, $m_2$, $n_1$, $n_2$ are positive integers, they
  satisfy the following constraints

  \begin{equation}\label{constrr}
     m_ \geq 1, \; m_2 \geq 1, \;  n_1 \geq 1, \; n_2 \geq 1, \; m_1 + m_2 \geq r, \; n_1 + n_2 \geq r.
 \end{equation}

 \begin{theo}\label{e=rk+r-1}
   Let $2 \leq r \in \N$, $e \in [r^2+1, \infty)$ and consider
   the minimum of $\omega = m_1 m_2 n_1 n_2$ subject to
   $m_1$, $m_2$, $n_1$, $n_2$ $ \in \R$,
   (\ref{constre}) and (\ref{constrr}). Then the minimum
   is $\frac{(r-1)(e-r+1)}{r}$,
   and it is achieved only in one of the two cases
   \begin{eqnarray}\label{minsol3k+r-1}
     (m_1, m_2) = (r-1,1), \; (n_1, n_2) = \left(\frac{e-r+1}{r},1\right)\\
     (m_1, m_2) = \left(\frac{e-r+1}{r},1\right), \; (n_1,n_2) = (r-1,1).
   \end{eqnarray}
 \end{theo}

 \begin{proof}
   Let $\overline \omega$ be the minimum value of $\omega$
   subject (\ref{constre}) and (\ref{constrr}). Note that since
   for the values of $m_1$, $m_2$, $n_1$, $n_2$ given in
   (\ref{minsol3k+r-1}) we have $\omega =
   \frac{(r-1)(e-r+1)}{r}$, we deduce that $\overline \omega
   \leq \frac{(r-1)(e-r+1)}{r}$. Since all the functions are
   symmetric in $(m_1, m_2)$ and $(n_1, n_2)$, we will always
   assume that $m_1 + m_2 \leq n_1+n_2$. As $m_1 \geq 1$ and
   $m_2 \geq 1$, we have $(m_1-1)(m_2-1) \geq 0$, which implies
   $m_1 m_2 \geq m_1 + m_2 - 1$. Similarly $n_1 n_2 \geq n_1 +
   n_2-1$. Thus
   \[\omega = m_1 m_2 n_1 n_2 \geq (m_1 + m _2 - 1)(n_1 + n_2 - 1)\geq
   (m_1 + m_2 - 1)^2.\]
   Suppose that $m_1 + m_2 \geq \frac{e}{r}$.
   Then
   \[ \omega \geq (m_1 + m_2 - 1)^2 \geq
   \left( \frac{e}{r} - 1
   \right) ^2 = \frac{(e-r)^2}{r^2}.\]
   Since $e \geq r^2+1$
   and $r \geq 2$, we obtain
   \[
     (e - r)^2 - r(r -1 )(e - r +1) =
     (e - r^2 - 1)(r^2 - r + 2) + 1 \geq 0 \cdot 0 + 1.
   \]
   Hence $\omega > \frac{(r - 1)(e - r + 1)}{r}$ if $m_1 + m_2 \geq \frac{e}{r}$.
   Thus to find the value of $\overline\omega$, we may assume that
   $m_1 + m_2 < \frac{e}{r}$.

   Fix $m_1$, $m_2$ that satisfy the conditions
   \begin{equation}\label{newconm1m1}
     m_1 \geq 1, \; m_2 \geq 1, \; \frac{e}{r} > m_1 + m_2\geq r.
   \end{equation}
   Now let us find the minimum of $n_1n_2$ subject to
   $n_1 \geq 1, \; n_2 \geq 1$ and (\ref{constre}).
   The constraint (\ref{constre}) is equivalent
   to (\ref{auxminc}) with $a = m_1 + m_2$, $b = m_1$, $x = n_1$ and $y = n_2$.
   Also $a + b = 2m_1 + m_2 < r(m_1 + m_2) \leq e$.
   Since $a > b$, Lemma~\ref{auxmin} implies that
   $n_1n_2$ is at a minimum when
   \begin{equation}\label{optvn1n2}
     n_2 = 1, \; n_1 := n_1(e,m_1,m_2) = \frac{e-b}{a} = \frac{e-m_1}{m_1+m_2}.
   \end{equation}
   Clearly
   \[n_1(e,m_1,m_2) > \frac{e - (m_1+m_2)}{m_1+m_2} > \frac{e-\frac{e}{r}}{\frac{e}{r}} = r-1 \geq
   1,\]
   and hence $n_1(e,m_1,m_2) + 1 \geq r$.
   Thus the problem of minimizing $m_1 m_2 n_1 n_2$ subject to
   (\ref{constre}) and (\ref{constrr}) is equivalent to the problem
   \begin{equation}\label{newminpr}
     \text{minimize }
     \tau := m_1 m_2 n_1 (e,m_1,m_2) = \frac{(e-m_1)m_1m_2}{m_1+m_2} \;
     \text{ subject to (\ref{newconm1m1})}.
   \end{equation}

   Fix $m_1$.  Since the function $\frac{t}{m_1 + t}$ increases
   for $t > 0$, the minimum of $\tau$ in (\ref{newminpr}) is
   achieved only
   in the following two cases:
   \begin{enumerate}
     \item\label{case1} Case~1: $m_2 = r - m_1$ if $1 \leq
         m_1 \leq r-1$;
     \item\label{case2} Case~2: $m_2 = 1$ if $m_1 \geq
         r-1$.
   \end{enumerate}

   Consider first Case~\ref{case1}.
   Assume first that $r = 2$.  Then $m_1 = m_2 = 1$,
   and the minimum of $\tau$ subject to Case~\ref{case1}
   is $\frac{(e-1)}{2}$.

    Now assume that $r \geq 3$. Then $\tau = f(m_1)$, where
    $f(t) := \frac{t(r-t)(e-t)}{r}$.  Note that $f(t)$ is a
    cubic with zeros at $0$, $r$ and $e$.  Also $f(t) < 0$ for
    $t < 0$ and $f(t) > 0$ for $t > e$.  Hence $f'(t) = 0$ for
    some $t = $ $t_1 \in (0,r)$ and some $t = t_2 \in (r,e)$.
    Note that $f(t)$ increases on $(-\infty, t_1)$ and
    decreases on $(t_1,t_2)$.

    We need to find $\min_{t \in [1,r-1]} f(t)$.
    Clearly
    \[rf'(t) = (r-t)(e-t) - t(e-t) - t(r-t) = (r-2t)(e-2t) - t^2.\]
    As $(r-2)(e-2) \geq e-2 \geq r^2-1 \geq 8$, we deduce that $f'(1) > 0$.
    As $r-2(r-1) = 2 - r < 0$ and
    $(e-2r+2) \geq (r^2-2r+3) = (r-1)^2+2 > 0$,
    it follows that $f'(r-1) < 0$.  So $1 < t_1 < r-1$.  Hence
    \[\min_{t \in [1,r-1]} f(t) =
    \min(f(1),f(r-1)) = f(r-1) = \frac{(e-r+1)(r-1)}{r},\]
    which is achieved only for $m_1 = r-1$ and $m_2 = 1$.

    We now consider Case~\ref{case2}.  In this case $\tau = g(m_1)$, where
    \[g(t) = \frac{(e-t)t}{t+1} = e + 1 - t -\frac{e+1}{t+1}.\]
    Thus on $[0,\infty)$, the function $g'(t)$ vanishes at the
    point
    \[t_0 = \sqrt{e+1} - 1 > \sqrt{r^2} - 1 = r-1.\]
    Note that $t_0$ is the unique
    solution of (\ref{constre}), where $m_1 = n_1 = t$
    and $m_2 = n_2 = 1$.
    So $f(t)$ increases on $[0,t_0]$ and decreases on the interval
    $[t_0,\infty)$.  Our assumption that $m_1 + m_2 = m_1+1 \leq
    n_1+n_2 = n_1+1$ is
    equivalent to $m_1 \leq n_1$.  So $m_1\leq t_0 \leq n_1$.  The
    assumption that $m_1+m_2 = m_1+1 \geq r$ means that $m_1 \geq r-1$.
    Hence $g(m_1)$ has a unique minimum at $m_1=r-1$.  This
    corresponds to $n_1 = \frac{e-r+1}{r}$.
    Interchanging $(m_1,m_2)$ with $(n_1,n_2)$ obtain the second
    solution $(m_1,m_2) = (\frac{e-r+1}{r},1),\;(n_1,n_2)=(r,1)$.
  \end{proof}

  \begin{theo}\label{cor:thme=rk+r-1}
    Suppose one of the following conditions holds.
    \begin{enumerate}
      \item $r = 2$, $e \geq 3$ is odd and $2 \leq p \leq
          q$, $l = \frac{e-1}{2} < q$.
      \item $3 \leq r \in \N$ and  (\ref{asmrcon1}) holds.
    \end{enumerate}
    Then $\min_{G \in \cK_2(p,q,e)} \lm (G)$
    is achieved only for $G_D$ isomorphic to the graph obtained from
    $K_{r-1,l+1}$ by adding one vertex to the group of $r-1$ vertices
    and connecting it to $l$ vertices in the group of $l+1$ vertices.
  \end{theo}

  \begin{proof}  Assume that $G_D$ has the Ferres diagram given
    in Figure~\ref{fig:Ferrers} with $h=2$.  Let
    $n_1 = r_2$ and $n_2 = r_1 - r_2$.
    Assume first that $r=2$ and $3 \leq e$ is odd.  Then
    $m_1 \geq 1$, $m_2 \geq 1$, $n_1 \geq 1$ and $n_2 \geq
    1$, so $m_1 + m_2 \geq 2$ and $n_1 + n_2 \geq 2$.  Then for $e \geq 5$ the
    theorem follows from Theorem~\ref{e=rk+r-1} for $r=2$ and
    Theorem \ref{theo:estimate}.  For $e=3$ the theorem is trivial.
    For $r \geq 3$  the theorem follows from Proposition~\ref{prop:upbndm},
    Theorem~\ref{e=rk+r-1} and  Theorem~\ref{theo:estimate}.
  \end{proof}

  \section{A special case of Problem~\ref{prob:minint}}

  We now discuss a
  special case of Problem~\ref{prob:minint}
  that is not covered by Theorem~\ref{e=rk+r-1}.

  \begin{theo}\label{e=3k+1}
    Let $e = 3k+1$, where $k \geq 7$ is an integer. Consider
    the minimum of $\omega = m_1 m_2 n_1 n_2$, where $m_1$,
    $m_2$, $n_1$ and $n_2$ are positive integers satisfying the
    constraints $m_1(n_1+n_2) + m_2n_1 = e$, $m_1+m_2 \geq 3$,
    $n_1+n_2 \geq 3$; in other words, (\ref{constre}) and
    (\ref{constrr}) with $r=3$ hold. Then the minimum of
    $\omega$ is $2k$, and it is achieved if and only one of the
    following cases holds:
    \begin{equation}
    \label{minsol3l+1}
      (m_1,m_2) = (1,2), \; (n_1,n_2) = (k,1)
    \end{equation}
    \begin{equation}
      \label{minsol3l+1prime}
      (m_1,m_2) = (k,1), \; (n_1,n_2) = (1,2).
    \end{equation}
  \end{theo}

  \begin{proof}  Clearly, if
    (\ref{minsol3l+1}) or (\ref{minsol3l+1prime}) holds, then
    $\omega = 2k$.  Thus it is enough to show that for all
    integer values of $m_1$, $m_2$, $n_1$ and $n_2$ satisfying
    the constraints that are different from the values given
    in~(\ref{minsol3l+1}) and (\ref{minsol3l+1prime}), we have
    $\omega > 2k$.

    Since we can interchange $m_1$ with $n_1$ and $m_2$ with
    $n_2$, we assume without loss of generality that $n_1 + n_2
    \geq m_1 + m_2$. We denote the product $m_1m_2$ by $X$.
    Since $m_1 + m_2 \geq 3$, it follows that $X \geq 2$.

    {\bf Case 1:} $m_1 + m_2 \geq k$. Since $n_1 + n_2 \geq m_1
    + m_2 \geq k$, we have $m_1m_2 \geq k-1$ and $n_1n_2 \geq
    k-1$. This implies  $\omega = m_1m_2n_1n_2 \geq (k-1)^2 >
    2k$ since $k\geq 7$.

    {\bf Case 2:} $3 < m_1 + m_2 < k$.  Hence $X \geq 3$.
    Suppose $X \geq k-1$. Since $n_1+n_2 \geq 4$, we have
    $n_1n_2 \geq 3$. Thus $\omega = m_1m_2n_1n_2 \geq 3(k-1) >
    2k$ since $k\geq 7$.

    So for the remaining part of Case~2 we assume that $X <
    k-1$, and hence  $\floor{\frac {2k}{X}} \geq 2$. We will
    now show that $n_1n_2 \geq \floor{\frac {2k}{X}} + 1$, from
    which it will follow  that $\omega = X n_1 n_2
    >X {\frac{2k}{X}} = 2k$, as required.

    Assume to the contrary that $n_1 n_2 < \floor{\frac {2k}
    {X}} + 1$. Since $n_1 n_2\leq \floor{\frac {2k}{X}}$ we
    have $n_1 + n_2 \leq \floor{\frac{2k}{X}} + 1$, and hence
    $n_1 \leq \floor{\frac {2k}{X}}$. We now obtain an upper
    bound for $e$.

    Clearly
    \[e = m_1(n_1+n_2) + m_2n_1
    \leq  m_1 \left( \floor{{\frac{2k}{X}}} + 1 \right) + m_2n_1
    \leq  m_1 \left( \floor{{\frac{2k}{X}}} + 1 \right) + m_2 \floor{\frac{X}{2k}}.
    \]
    Observe that that for any $0 < a \in \R$ we have the inequality
    \[m_1(a+1) + m_2a \leq m_1m_2(a+1) + a = X(a+1) + a\]
    and hence
    \begin{equation}\label{funine}
      e \leq X \left( \floor{ {\frac{2k}{X} } } + 1 \right)
      + \floor {\frac{2k}{X}}
      = X \floor {\frac{2k}{X}} + X + \floor {\frac{2k}{X}}.
    \end{equation}

    Let $f(x) = x + \frac {2k}{x}$.  Since $f$ is strictly
    convex on $(0,\infty)$, it follows that
    %
    %
    $f(X) \leq \max(f(3),f(k-2)) = \max ( 3 + \frac {2k}{3}, k - 2 +
    \frac{2k}{k-2} ).$
    Hence  $f(X) < k+1$ for $k > 6$. We also note that in the
    range $[3,k-2]$, $f(x)$ has a unique minimum at $x = \sqrt
    {2k}$.  Furthermore, $f(x)$ is strictly increasing after
    that point.

    Consider the function $g(x) = x + \floor{\frac {2k}x}$ on
    the interval $x \in (0,\infty)$.   Note that $g(x)$  is
    piecewise linear, where each linear piece has slope 1, and
    is continuous from the left, with jumps at $x_i =
    \frac{2k}{i}$ for $i \in \N$.  So $g(x_i) = f(x_i)$ for $i
    \in N$, and $g(x) <f(x) $ if $0 < x \neq x_i$ for $i \in
    N$. This implies $g(X) = X + \floor{\frac {2k} X} \leq
    f(X)< k+1$ for $k > 6$. Since $g(X)$ is an integer, we
    deduce that $g(X) \leq k$ when $k > 6$ and $X \in [3,k-2]$.
    Use (\ref{funine}) to deduce that
    \[e\leq X \floor{\frac {2k}{X}} + g(X)
    \leq X \frac {2k}{X} + g(X) \leq 2k+k = 3k.\]
    This contradicts the assumption that $e=3k+1$,
    and completes Case~2.

    {\bf Case 3:}
%
$m_1+m_2=3$.
    Then $m_1m_2 = 2$ and $e = 3n_1 + m_1n_2$. Assume first
    that $m_1 = 2$ and $m_2=1$. Since $e = 3k + 1$ it follows
    that $n_2 \geq 2$.  So $3n_1n_2 = n_2(e-2n_2)$.  Since $e
    \geq 22$, the minimum of $n_1n_2$ is achieved either for
    $n_2 = 2$ or for the maximum possible value of $n_2$
    obtained when $n_1 = 1$. For $n_2 = 2$ we have $n_1 = k-1$
    and  $\omega = 4(k-1) > 2k$ if $k > 2$. For $n_1 = 1$ we
    have $n_2 = \frac{e-3}{2}$, which may not be an integer,
    and $\omega = (e-3) = 3k-2 > 2k$ for $k>2$.

    Assume finally that $m_1 = $ and $m_2 = 2$.  Then $e = 3n_1
    + n_2$.  Lemma~\ref{auxmin} yields that $n_1n_2 \geq
    \frac{e-1}{3} = k$, and equality holds if and only if $n_2
    = 1$ and $n_1 = k$. This completes Case~3 and the proof of
    the theorem.
  \end{proof}

  We used software to show that Theorem \ref{e=3k+1} holds
  for $k=2,3,4,5,6$.

 \section{C-matrices}

 Let $\R^p_{+\searrow} := \{\c = (c_1,\ldots,c_p) \in \R^p, \; c_1 \geq \cdots \geq
 c_p \geq 0\}$.  With each $\c \in \R^p_{+\searrow}$ we associate the
 following symmetric matrix.
 \begin{equation}\label{eq:defmx}
    M(\c) = [c_{\min(i,j)}]_{i,j=1}^p.
 \end{equation}
 The following result is
 well-known \cite[\S3.3, pp.110--111]{Kar}.
 \begin{prop}\label{prop:tpmx}
   Let $\c = (c_1, \ldots, c_p) \in \R^p_{+\searrow}$.
   Then all the minors of $M(\c)$ are nonnegative.  In
   particular $M(\c)$ is a nonnegative definite matrix.
   If $c_1 > \cdots > c_p > 0$, then all the principal minors of $M(\c)$ are
   positive, i.e., $M(\c)$ is positive definite.
 \end{prop}

 \begin{corol}\label{cor:tpmx}
   Let $\c = (c_1,\ldots,c_p) \in \R^p_{+\searrow}$.
   Then the rank of $M(\c)$ is equal to the number of
   distinct positive elements in $\set{c_1,\ldots,c_p}$.
 \end{corol}

 \begin{proof}
   Let $\set{c_{i_1},\ldots,c_{i_k}}$ be the set of all
   distinct positive elements in $\set{c_1,\ldots,c_p}$.
   Hence the rank of $M(\c)$ is at most $k$.
   Let $F$ be the principal submatrix of $M(\c)$
   based on the rows and columns $\set{i_1,\ldots,i_k}$.
   Proposition~\ref{prop:tpmx} yields that rank $F = k$.
 \end{proof}

 In what follows we assume that $\c = (c_1,\ldots,c_p)\in
 \R^p_{+\searrow}$ unless stated otherwise. Assume that $c_1
 \geq \cdots \geq c_m > 0 = c_{m+1} = \cdots = c_p$.  Denote by
 $\c_+ := (c_1,\ldots,c_m)$.  Then $M(\c_+)$ is the principal
 submatrix of $M(\c)$ obtained from $M(\c)$ by deleting the
 last $p-m$ zero rows and columns. Let
 \[\lambda_1(\c) \geq \lambda_2(\c) \geq \cdots \geq \lambda_p(\c) \geq 0\]
 be the $p$ eigenvalues of $M(\c)$. Let $m'$ be the number of
 distinct elements in $\{c_1,\ldots,c_m\}$.
 Corollary~\ref{cor:tpmx} yields that $M(\c)$ has exactly $m'$
 positive eigenvalues and $\lambda_i(\c) = \lambda_i(\c_+)$ for
 $i = 1,\ldots,m$.

 Let $\rS_m(\R)$ be the space of $m \times m$ real symmetric
 matrices. Since $\lambda_1(M)$ is a convex function on $\rS_m(\R)$
 by
 \[\max_{\norm{\z} = 1 } \z \trans \frac{A+B}{2} \z
 \leq \frac{\max_{\norm{\x} = 1} \x \trans A \x + \max_{\norm{\y} = 1} \y \trans B \y}{2},\]
 we obtain the following result.

 \begin{prop}\label{prop:convp}
   Let $\rC \subset \R^p_{+\searrow}$
   be a compact convex set.  Let $\cE(\rC)$ be the set of the extreme
   points of $\rC$.  Then $\max_{\c \in \rC} \lambda_1(\c) = \max_{\c \in \cE(C)}
   \lambda_1(\c)$.
 \end{prop}

 We show that
 $\lambda_1(G_D) = \lambda_1(\d)$ for a corresponding vector $\d \in
 \R^p_{+\searrow}$.  Let $D = \set{d_1 \geq d_2 \geq \cdots \geq d_m}$
 be a set of positive integers.  Assume that $m \leq p$ and let $\d = (d_1,\ldots,d_m,0,\ldots,0) \in
 \R^p_{+\searrow}$.  Let $G(\d)$ denote the chain graph
 with degrees $\d$, that is to say the chain graph $G_D$ with $p-m$
 additional isolated vertices, where $D = \set{d_1,\ldots,d_m}$.
 Let $A(\d)$ be the representation matrix of $G(\d)$.  Note that $A(\d_+)$
 is the representation matrix of $G_D$.  Clearly $A(\d)A(\d)\trans
 = M(\d)$.  Hence

 \begin{equation}\label{eqMd}
   \lm(G_D)^2 = \lambda_1(\d) = \lambda_1(\d_+).
 \end{equation}

 Thus $M(\d)$ can be viewed as a continuous version of $G(\d)$.
 The main idea of the proof of Conjecture~\ref{bipcone},
 under the conditions discussed in the Introduction,
 is to replace the maximum discussed in Problem~\ref{bipprob}
 with the maximization problem discussed in Proposition~\ref{prop:convp}
 with a carefully chosen $\rC$.

 We now bring a few inequalities for $\lambda_1(\d)$ needed later,
 which can be viewed as a generalizations of Proposition~\ref{upestimlm}
 and Theorem~\ref{theo:estimate}.

 \begin{prop}\label{basestl1e}
   Let $\c=(c_1.\ldots,c_p) \in \R^p_{+\searrow}$.  Then

   \begin{eqnarray}\label{trace1}
     && e(\c) : = \tr M(\c) = \sum_{i=1}^p c_i = \sum_{i=1}^p \lambda_i(\c),\\
     && \sum_{i=1}^p \lambda_i(\c)^2 = \tr M(\c)^2 = \sum_{i=1}^p (2i-1)c_i^2,
     \label{trace2}\\
     && \sum_{1 \leq i < j \leq p} \lambda_i(\c) \lambda_j(\c) = \sum_{1 \leq i <
     j \leq p} c_j(c_i-c_j). \label{2elpol}
   \end{eqnarray}
   Hence $\lambda_1(\c) \leq e(\c) $.  Equality holds if and only if
   $M(\c)$ has rank one.  Moreover,
   \begin{equation}\label{est1}
     \lambda_1(\c) \leq \sqrt{\sum_{i=1}^p (2i-1)c_i^2}.
   \end{equation}
   A sharper upper estimate of $\lambda_1(\c)$ for $p \geq 2$ is given
   as follows.  Assume that the set $\set{c_1,\ldots,c_p}$
   consists of $h \geq 2$ distinct positive numbers.  Then
   \begin{gather}\label{maxest}
     \lambda_1(\c) \leq
     \frac{(2\alpha_h-1)e(\c) + \sqrt{ e(\c)^2-4\alpha_h\beta}}{2\alpha_h},
     \\
     \intertext{where}
     \alpha_h = \frac{h}{2(h-1)},
     \qquad
     \beta = \sum_{1 \leq i < j \leq p} c_j(c_i-c_j).
     \notag
   \end{gather}
   For a fixed $\beta$, the right-hand side of~(\ref{maxest}) is an
   increasing sequence for $h = 2,3,\ldots$, and its limit as $h \to
   \infty$ is the right-hand side of~(\ref{est1}).
 \end{prop}
 \begin{Proof}
   The equalities~(\ref{trace1}), (\ref{trace2}) are
   straightforward.  The equality (\ref{2elpol}) follows from
   them and the identity
   \begin{equation}\label{newtid2}
     2\sum_{1 \leq  i < j \leq p} \lambda_i \lambda_j =
     \left(\sum_{i=1}^p \lambda_i\right)^2 - \sum_{i=1}^p \lambda_i^2.
   \end{equation}
   Since all $\lambda_i(\c)$ are real, the inequality (\ref{est1})
   follows from~(\ref{trace2}).


   We now show how the estimate (\ref{maxest}) follows
   from~(\ref{2elpol}). By Corollary~\ref{cor:tpmx} the rank of $M(\c)$ is
   exactly $h$, and therefore exactly $h$ eigenvalues of $M(\d)$ are
   positive. Thus in the left-hand side of the
   equalities (\ref{trace1})--(\ref{2elpol}), $i$ and $j$ can run from
   $1$ to $h$ only.
   For $h=2$, (\ref{maxest}) follows with equality
   from~(\ref{trace1}) and (\ref{2elpol}).

   Assume that $h \geq 3$.  We let
   $\lambda_1(\c) = x$ and $e(\c) = e$ and rewrite (\ref{2elpol}) as follows:
   \begin{equation}\label{separatethex}
     x(e-x) = \beta - \sum_{2 \leq i < j \leq h} \lambda_i(\c)\lambda_j(\c).
   \end{equation}
   In~(\ref{separatethex}) we are free to choose
   $\lambda_2(\c),\ldots, \lambda_h(\c)$ subject to $x + \sum_{2 \leq
   i \leq h} \lambda_i(\c) = e$, and wish to maximize $x$.

   Observe that the right-hand side of (\ref{separatethex}) is
   always nonnegative by (\ref{2elpol}) and the definition of
   $\beta$. Thus~(\ref{separatethex}) has two solutions $x$
   between $0$ and $e$, and we are interested in the larger one
   and want to maximize it. This is equivalent to choosing
   $\lambda_2(\c),\ldots, \lambda_h(\c)$ so as to minimize the
   right-hand side of (\ref{separatethex}), or equivalently to
   maximize $\sum_{2 \leq i < j \leq h}
   \lambda_i(\c)\lambda_j(\c)$. It is well-known that a sum of
   the form $\sum_{1 \leq i < j \leq k} a_i a_j$ can only
   increase if the $a_i$ are each replaced by their arithmetic
   mean $\frac{\sum a_i}{k}$. Indeed, $(\sum 1 \cdot a_i)^2
   \leq (\sum 1^2) (\sum a_i^2)= k \sum a_i^2$. Therefore $2
   \sum_{\i < j} a_i a_j = (\sum a_i)^2 - \sum a_i^2 \leq (\sum
   a_i)^2 - \frac{1}{k} (\sum a_i)^2 = (\sum a_i)^2
   \frac{k-1}{k} = 2 \binom{k}{2} \left(\frac{\sum a_i}{k}
   \right)^2$.

   Thus the upper estimate on $x$ is achieved if we set each of
   $\lambda_2(\c), \ldots, \lambda_h(\c)$ equal to $y$. Then
   $x$ and $y$ are subject to $x+(h-1)y = e$ and  $x(e-x) +
   \binom{h-1}{2} y^2 = \beta$. Eliminating $y$, we see that
   $x$ should satisfy
   \begin{equation}\label{quadonx}
     x(e-x) + \frac{h-2}{2(h-1)} (e-x)^2 = \beta,
   \end{equation}
   and the larger solution of~(\ref{quadonx}) yields (\ref{maxest})
   for $h \geq 3$.

   The left-hand side of~(\ref{quadonx}) is a quadratic in $x$,
   which is positive for $0 < x < e$ and increases with $h$.
   Therefore the larger solution of~(\ref{quadonx}) increases
   with $h$. When we take~(\ref{quadonx}) to the limit $h \to
   \infty$, we obtain $x(e-x) + \frac{1}{2}(e-x)^2 = \beta$, or
   equivalently
   \begin{equation}\label{limitingequation}
     x^2 = e^2 - 2\beta = \sum_{i=1}^p (2i-1)c_i^2,
   \end{equation}
   and the positive solution of~(\ref{limitingequation})
   is equal to the right-hand side of~(\ref{est1}).
 \end{Proof}


 \section{A proof of Conjecture~\ref{bipcone} in certain cases}

 \begin{theo}\label{thm: prfmconj}
   Let $2 \leq r \leq l$ be two positive integers.  Assume that
   $e = rl+r-1$.
   Suppose one of the following conditions holds:
     \begin{enumerate}
       \item $r = 2$ and $2 \leq p \leq q$, $l =
           \frac{e-1}{2} < q$;
       \item $3 \leq r \leq l$ and $r \leq p \leq l+1 \leq
           q \leq l + 1 + \frac{l}{r-1}$.
     \end{enumerate}
   Let $G_{r,l+1}$ be the graph obtained from
   $K_{r-1,l+1}$ by adding one vertex to the group of $r-1$ vertices
   and connecting it to $l$ vertices in the group of $l+1$ vertices.
   Then $\lm(G) \leq \lm(G_{r,l+1})$, for all $G \in \cK(p,q,e)$.
   Equality holds if and only if $G$ is isomorphic to $G_{r,l+1}$.
 \end{theo}

 \begin{proof}
    Corollary~\ref{cor:optsolmpr} implies that in order to find
    $\max_{G \in \cK(p,q,e)}$, it is enough to consider graphs
    $G_D = (U,V,E)$, for some $D  = \set{d_1 \geq d_2 \geq
    \cdots \geq d_m}$, where $m \leq p$ and $d_1 \leq q$.  We
    are going to assume that $\#U = m \leq \#V = d_1$  (if this
    is not satisfied consider the isomorphic graph
    $G_D'=(V,U,E)$). Let $\d = (d_1,\ldots,d_m)$ be the degree
    sequence of $D$.

    Since $\cK(e,p,q)$ does not contain a complete bipartite
    graph, we know that $m \geq 2$.
    Proposition~\ref{prop:upbndm} yields that $m \geq r$ for $r
    \geq 3$. Let $\delta_i := d_i - d_m$ for $i=1,\ldots,m$.
    We define $s$ by $\delta_s > 0 = \delta_{s+1} = \cdots =
    \delta_m$, and put $\bm{\delta} :=
    (\delta_1,\ldots,\delta_s)\trans$. Thus switching from $\d$
    to $\bm{\delta}$ amounts to deleting the first $d_m$
    columns of the Ferrers diagram of $\d$, and then deleting
    empty rows. The resulting Ferrers diagram has a total of
    $e' := e - md_m = \sum_{i=1}^s \delta_i$ components equal
    to $1$ and the rest are zero. Note that $\delta_1 \leq q -
    d_m$. As $h \geq 2$, we have that $\delta_s \geq 1$.  Hence
    $e' \leq e-1$ and $\frac{e'}{s} \geq 1$.

    We now consider the following polyhedron in $\R^s$:
    \begin{equation}\label{bigcone}
      P := \set{(x_1,\ldots,x_s)\trans \in \R^s, \;\; x_1 \geq x_2 \geq
      \cdots \geq x_s \geq 0, \;\; \sum_{i=1}^s x_i = e'}.
    \end{equation}
    Using the notation
    \begin{equation}\label{def1li}
      \1_{n,i} := (\underbrace{1,\ldots,1}_i,0,\ldots,0)\trans\in \R^n, \quad
      i=1,\ldots,n,
    \end{equation}
    it is clear that the extreme points of $P$ are $\frac{e'}{i}
    \1_{l,i}$, $i=1,\ldots,s$.

    Let us define
    \begin{equation}\label{defai}
      \a_{i}(\d) = (a_{1,i},\ldots,a_{m,i})\trans := \frac{e'}{i}\1_{m,i} +
      d_m\1_{m,m}, \quad i=1,\ldots,s.
    \end{equation}
    We note that $\d \in \rC(\d) := \conv{\set{\a_{1}(\d), \ldots,\a_{s}(\d)}}$. Indeed,
    $\bm{\delta} \in P$, and therefore there exist
    $\alpha_1,\ldots,\alpha_s \geq 0$ satisfying $\sum_{i=1}^s \alpha_i
    = 1$ and $\sum_{i=1}^s \alpha_i \frac{e'}{i}\1_{l,i} =
    \bm{\delta}$. Then
       \begin{eqnarray*}
        \sum_{i=1}^s \alpha_i \a_{i}(\d) =
        \sum_{i=1}^s
        \alpha_i \left( \frac{e'}{i} \1_{m,i} + d_m \1_{m,m} \right) =
        \left( \sum_{i=1}^s\alpha_i\frac{e'}{i}\1_{m,i} \right) + d_m\1_{m,m}
        \\
        (d_1,\ldots,d_s,d_m,\ldots,d_m)\trans =
        (d_1,\ldots,d_s,d_{s+1},\ldots,d_m)\trans = \d.
      \end{eqnarray*}
     Since $\d$ is a convex combination of $\a_1(\d),\ldots,\a_s(\d)$, it
     follows that $M(\d)$ is the same convex combination of
     $M(\a_1(\d)),\ldots,M(\a_s(\d))$. Combine (\ref{eqMd}) with
     Proposition~\ref{prop:convp}
     to obtain
     \begin{equation}\label{basicin}
       \lm(G_D)^2 = \lambda_1(M(\d)) \leq \max_{1 \leq k \leq l} \lambda_1(M(\a_k(\d))).
     \end{equation}
     The vector $\a_k(\d)$ has the form
     $\a_k(\d) = (\underbrace{x,\ldots,x}_k,\underbrace{y,\ldots,y}_{m-k})\trans$
     with $x = \frac{e'}{k} +
     d_m$ and $y = d_m$. Therefore the first $k$ rows of $M(\a_k(\d))$ are
     equal to $\a_k(\d)$ and the last $m-k$ rows are equal to $y \1_{m,m}$.
     Proposition~\ref{prop:tpmx} and Corollary~\ref{cor:tpmx} yield that $M(\a_k(\d))$
     is a nonnegative definite matrix of rank $2$.
     It satisfies
     \begin{equation}\label{eigmak}
       \lambda_1(M(\a_k(\d))) + \lambda_2(M(\a_k(\d))) = \tr M(\a_k(\d)) =
       k \left(\frac{e'}{k}+d_m \right) + (m-k)d_m = e.
     \end{equation}
     The product $\lambda_1(M(\a_k(\d)))\lambda_2(M(\a_k(\d)))$ is equal
     to the sum
     \[\sum_{i<j} \lambda_i(M(\a_k(\d))) \lambda_j(M(\a_k(\d))),\]
     which in turn equals the
     coefficient of $\lambda^{m-2}$ in the characteristic polynomial
     \[\prod_i (\lambda - \lambda_i(M(\a_k(\d)))).\]
     This coefficient equals in turn the sum of all $2 \times
     2$ principal minors of $M(\a_k(\d))$. There are $k(m-k)$
     contributing minors, each of the form $\det
     \twobytwo{x}{y}{y}{y} = y(x-y)$. Therefore
     \begin{equation}\label{eigmakprime}
       \lambda_1(M(\a_k(\d)))\lambda_2(M(\a_k(\d))) = \omega(\a_k(\d)) := k(m-k)\frac{e'}{k}d_m.
     \end{equation}
     Hence
     \begin{equation}\label{eq:eiginmak}
       \lambda_1(M(\a_k(\d))) = \frac{e + \sqrt{e^2-4 \omega(a_k(\d))}}{2}.
     \end{equation}
     Thus the maximum possible value of $\lambda_1(M(\a_k(\d)))$ is
     achieved for the minimum value of $\omega(\a_k(\d))$.
     This situation corresponds to the minimization problem we
     studied in Theorem~\ref{e=rk+r-1}. We put
     \[m_1 = k, \quad m_2 = m-k, \quad n_1 = d_m, \quad n_2 = \frac{e'}{k}.\]
     Clearly $m_1 \geq 1$ and $n_1 \geq 1$. Also $m_2 = m-k
     \geq s+1-k \geq 1$ and $n_2 = \frac{e'}{k} \geq
     \frac{e'}{s} \geq 1$. Recall that $m = m_1 + m_2 \geq r$.
     We claim that for each $2 \leq r \in \N$, the inequality
     $n_1+n_2 \geq r$ holds.  For $r=2$ this inequality follows
     from the inequalities $n_1 \geq 1, n_2 \geq 1$.  Assume
     now that $r \geq 3$.  Observe that
     \[n_1+n_2 = d_m + \frac{e'}{k} \geq d_m + \frac{e'}{s} \geq d_m + \frac{e'}{m-1}
     = d_m + \frac{e-md_m}{m-1} = \frac{e-d_m}{m-1}.\]
     Recall that $G_D$ is not complete bipartite.  Since $d_1 \geq \cdots \geq
     d_m$,
     we deduce that $d_m \leq \frac{e-1}{m}$.  Hence
     \[\frac{e-d_m}{m-1} \geq
     f(m) := \frac{e}{m-1} + \frac{1}{m(m-1)}.\]
     So $f(m)$ is a decreasing function for $m > 1$.  For $m =
     p = l+1$, obtain $f(l+1) > \frac{e}{l} = \frac{rl+r-1}{l}
     > r$ if $r \geq 3$. Hence $n_1+n_2 \geq r$ for any $r \geq
     2$. Theorem~\ref{e=rk+r-1} implies the inequality
     \begin{equation}\label{eq:maxomd}
       \omega(\a_k(\d)) \leq \frac{(r-1)(e-r+1)}{r}.
     \end{equation}
     Hence
     \[ \lambda_1(M(\a_k(\d))) \leq\frac{e+\sqrt{e^2-(\frac{(r-1)(e-r+1)}{r})^2}}{2}
     =\lm(G_{r,l+1})^2,\]
     where the last equality follows from Theorem~\ref{theo:estimate}.
     We use~(\ref{basicin}) to deduce
     \begin{equation}\label{eq:maxgrl}
        \lm(G_D)\leq \lm(G_{r,l+1}) \text{ for any } G_D \in \cK(p,q,e).
     \end{equation}
     It is left to show that equality holds
     in~(\ref{eq:maxgrl}) if and only if $D = D_* := \set{d_1 =
     \cdots = d_{r-1} = l+1 > d_r =l}$. Let
     $\d_*=(l+1,\ldots,l+1,1)$ be the corresponding degree
     sequence of $D_*$. We now consider the equality case
     in~(\ref{eq:maxomd}).  Theorem~\ref{e=rk+r-1} asserts that
     equality holds only if one of the two conditions
     in~(\ref{minsol3k+r-1}) holds.

     Assume first that the first condition
     of~(\ref{minsol3k+r-1}) holds.  So $n_1 = \frac{e-r+1}{r}
     = l$.  On the other hand $n_1 = d_m$.  So $d_m = l$.  Also
     $m = m_1+m_2 = r-1+1 = r$. Furthermore, $n_2 =
     \frac{e'}{k} = 1$.  This can happen if and only if
     $\delta_1 = \cdots = \delta_k=1$.  Hence $d_1 = \cdots =
     d_s = l+1$ and $d_{s+1} = d_r = l$.  Since $e = lr+r-1$,
     we deduce that $s=r-1$.

    Assume now that the second condition holds
    in~(\ref{eq:maxomd}). So $d_m = n_1 = r-1$ and
    $\frac{e'}{k} = 1$.  Hence $d_1 = \cdots = d_s = r >
    d_{s+1} = \cdots = d_m = r-1$. We have $m_1 =
    \frac{e-r+1}{r} = l = k$ and $m_2 = 1$.  So $m = m_1+m_2 =
    l+1$. Since $e = rl+r-1$, we deduce that $s = r-1$. Hence
    $D = D^* = \set{d_1 = \cdots = d_l = r > d_{l+1} = r-1}$.
    Note that $G_{D^*}$ is isomorphic to $G_{D_*} = G_{r,l+1} =
    (U_*,V_*,E)$. More precisely, $G_{D_*} = (V_*,U_*,E)$.
    Assume first that $r \geq 3$.  Then $\#V_* = l+1 > r =
    \#U_*$.  This case is ruled out since we agreed to consider
    only $G_D = (U,V,E)$ where $\#U \leq \#V$.  If $r=2$, then
    $\#V_* = \frac{e+1}{2} = l+1$ and $\#U = 2$. If $e \geq 5$,
    then this case is ruled out as above. If $e = 3$, then any
    $G \in \cK(2,2,3)$ is isomorphic to $G_{2,2}$, and the
    theorem trivially holds in this case.  In particular any
    $G_D \in \cK(2,2,3)$ is equal to $G_{D_*}$.

    Let $G_D = (U,V,E) \in \cK(p,q,e)$, and assume that $\#U
    \leq \#V$ and $D \ne D_*$.  The above arguments show that
    $\omega(\a_k(\d)) > \omega(\a_{r-1}(\d_*))$ for
    $k=1,\ldots,s$.  Hence $\lambda_1(M(\a_k(\d)) <
    \lm(G_{r,l+1})^2$ for $k=1,\ldots,s$, and~(\ref{basicin})
    yields that $\lm(G_D) < \lm(G_{r,l+1})^2$.
 \end{proof}

   \bibliographystyle{plain}
   
 \end{document}